\documentclass[11pt, reqno]{amsart}
\usepackage[colorlinks=true,pagebackref,hyperindex]{hyperref}
\usepackage{mathrsfs} 
\usepackage[all]{xy}
\usepackage{color}
\usepackage{amsfonts}
\usepackage{amscd}
\usepackage{amssymb,latexsym,amsmath,amscd, graphicx} %Mircea's

\definecolor{darkgreen}{rgb}{0.0, 0.4, 0.0}

\definecolor{cyan}{cmyk}{1,0,0,0}

\newcommand{\bdg}{\begin{dg}}
\newcommand{\edg}{\end{dg}}

\newcommand{\cre}{\color{red}}

\newtheorem{tm}{Theorem}[subsection]
\newtheorem{lm}[tm]{Lemma}
\newtheorem{pr}[tm]{Proposition}
\newtheorem{rmk}[tm]{Remark}
\newtheorem{cor}[tm]{Corollary}

\newtheorem{fact}[tm]{Fact}
\newtheorem{??}[tm]{Question}
\newtheorem{defi}[tm]{Definition}

\newtheorem{choice}[tm]{Choice}

\newcommand{\ben}{\begin{enumerate}}
\newcommand{\een}{\end{enumerate}}
\newcommand{\bit}{\begin{itemize}}
\newcommand{\eit}{\end{itemize}}
\newcommand{\beq}{\begin{equation}}
\newcommand{\eeq}{\end{equation}}
\newcommand{\la}{\label}

\newcommand\ci{\cite}

\font\tenmsb=msbm10
\font\sevenmsb=msbm7
\font\fivemsb=msbm5
\newfam\msbfam
\textfont\msbfam=\tenmsb
\scriptfont\msbfam=\sevenmsb
\scriptscriptfont\msbfam=\fivemsb
\def\Bbb#1{{\fam\msbfam #1}}
\font\teneufm=eufm10
\font\seveneufm=eufm7
\font\fiveeufm=eufm5
\newfam\eufmfam
\textfont\eufmfam=\teneufm
\scriptfont\eufmfam=\seveneufm
\scriptscriptfont\eufmfam=\fiveeufm
\def\frak#1{{\fam\eufmfam\relax#1}}

\newcommand{\lorw}{\longrightarrow}

\newcommand\rat{{\Bbb Q}}

\newcommand\comp{{\Bbb C}}

\newcommand\zed{{\Bbb Z}}

\newcommand{\pH}{\mathop{^{p}\mathcal{H}}\nolimits}

\newcommand\s{\sigma}

\newcommand\e{\epsilon}

\newcommand{\w}[1]{\widetilde{#1}}

\newcommand{\m}[1]{\mathcal{#1}}
\newcommand{\ms}[1]{\mathscr{#1}}

\newcommand{\ptd}[1]{ \,^{p}\!\tau_{ \leq {#1} } }
\newcommand{\ptu}[1]{ \,^{p}\!\tau_{ \geq {#1} } }

\newcommand{\ph}{^p\!H}
\newcommand{\Gm}{{\mathbb G}_m}
\newcommand{\disk}{\Delta}

\newcommand{\mo}{\ms{X}}

\title{The perverse  filtration for the Hitchin fibration is  locally constant}
\author{
Mark Andrea A.  de Cataldo, 
Davesh Maulik.
}

\address{Mark Andrea A. de Cataldo, Stony Brook University}
\email{mark.decataldo@stonybrook.edu}

\address{Davesh Maulik, Massachusetts Institute of Technology}
\email{maulik@mit.edu}

\begin{document}

\begin{abstract}
We prove that the perverse Leray filtration for the Hitchin morphism is locally constant in families, thus providing some evidence towards
the validity of the $P=W$ conjecture due to de Cataldo, Hausel and Migliorini in non Abelian Hodge theory.
\end{abstract}

\maketitle

\tableofcontents

$\,$
\section{Introduction}\la{intro}

Let $S$ be an algebraic variety, and let $f:X\to Y$ be an $S$-morphism of algebraic varieties.  Let $F$ be a constructible complex
of rational vector spaces on $X$.  For each $s \in S$, the vector spaces
$H^\bullet(X_s, F_s)$ carry the perverse Leray filtrations  $P^{f_s}$ associated with the morphism $f_s:X_s\to Y_s$.
It is natural to ask how these filtered vector spaces
$(H^\bullet(X_s, F_s),P^{f_s})$ behave as $s$ varies in $S$.   
The main goal of this paper is to provide criteria (Theorem \ref{finit}) for checking local constancy of these filtrations.

\subsection{Perverse filtration for the Hitchin morphism}\la{dffd}$\;$

Our motivating example arises from perverse filtrations associated to the Hitchin morphism for the moduli space of Higgs bundles for a smooth projective variety.  We refer to \S\ref{snaht} for more details on what follows.

Let $\mo \to S$ be a smooth projective morphism with connected fibers over a connected quasi projective variety, and let $G$ be a reductive group.
Associated to this family is the so-called Dolbeault moduli space $\pi_D=\pi_D(\mo/S,G): M_D(\mo/S,G) \to S$ and the  Hitchin projective  $S$-morphism $h=h(\mo/S,G): M_D(\mo/S,G)\to \m{V}(\mo/S,G)$, which induces, for every point $s\in S$,  the Hitchin morphism
$h_s=h(\mo_s,G): M_D(\mo_s,G)\to \m{V}(\mo_s,G)$.  

It follows from non-abelian Hodge theory that the 
the structural morphism $\pi_D$ is topologically locally trivial over $S$.
It follows that the intersection cohomology groups  $I\!H^\bullet (M_D(\mo_s,G), \rat)$ on the Dolbeault moduli spaces 
$M_D(\mo_s,G)$,
give rise to  local systems $I\!H^\bullet_D(\mo/S,G)$ on $S$.

For every $s\in S$, the Hitchin morphism $h_s$ induces the perverse Leray filtration
$P^{h_s}_{D,\star}  (\mo_s,G)$ on the intersection cohomology groups $I\!H^\bullet (M_D(\mo_s,G), \rat)$. 

It does not seem immediately clear that
these perverse Leray filtrations should match-up and give rise to locally constant  subsheaves of the local systems 
$I\!H^\bullet_D(\mo/S,G)$ on $S$.
A priori, the perverse filtrations could jump at special point in $S$.  For example, a priori, one could have un-expected direct summands in $D(\m{V}(\mo_s,G))$ in the decomposition theorem 
for the Hitchin morphism  $h(\mo_s,G)$ at  special  points $s \in S$. For a brief discussion, see the beginning of \S\ref{prep}.

In this paper, we prove the following result, which can be viewed as some evidence towards the validity of the $P=W$ Question and Conjecture, see Question
\ref{4r4r} and Remark \ref{p=wdd}: the weight filtration on the intersection cohomology of the Betti moduli spaces
of a family of projective manifolds gives rise to locally constant sheaves on the base of the family; if $P=W$, then the same would be true
for the corresponding perverse Leray filtrations associated with the family of Hitchin morphisms, and this is what the following theorem
asserts.

\begin{tm}\la{ewq0} {\rm  ({\bf The perverse Leray filtration  is locally constant})}
The perverse leray filtration gives rise to locally constant subsheaves:
\beq\la{wq}
P^h_{D,\star}  (\mo/S,G) \subseteq I\!H^\bullet_D(\mo/S,G).
\eeq
\end{tm}

Theorem \ref{ewq0} follows at once from Theorem \ref{ewq} below, to the effect that specialization is a filtered isomorphism
for the  perverse Leray filtrations. 

\begin{tm}\la{ewq} Let $S$ be a nonsingular connected curve, let $s\in S$ be any point, let $\disk$ be a suitably general disk
centered at $s$ and let $t \in \disk \setminus \{s\}$ be a suitably general point.
The specialization morphism $I\!H^\bullet (M_{D}(\mo_s,G),\rat)) \to I\!H^\bullet(M_{D}(\mo_t,G),\rat))$ is
well-defined and gives a filtered isomorphism for the respective perverse Leray filtrations $P^{h_s}_D(\mo_s,G)$ and $P^{h_t}_D(\mo_t,G).$ 
In particular, the dimensions
of the graded spaces ${\rm Gr}^{P^{h_s}_D}_\star I\!H^\bullet (M_{D}(\mo_s,G),\rat))$ are independent of $s \in S$. 
\end{tm}

\begin{rmk}\la{kkq}
Theorems \ref{ewq0} and  \ref{ewq} hold, with the same proof, if we replace the Dolbeault moduli spaces with their twisted counterparts
(cf. Remark \ref{p=wdd}) in the cases
of relative dimension $\dim \mo/S=1$ and $G=GL_n, SL_n, PGL_n$.  In these cases the Dolbeault moduli spaces are integral
orbifolds, even nonsingular in the cases $G=GL_n, SL_n$. 
\end{rmk}

We may say, informally, that the perverse Leray filtration for the Hitchin morphism associated with a family of projective manifolds
is independent of the members of the family. In particular, by applying this to curves and their moduli space, we may say that
the perverse Leray filtration on the intersection cohomology of the Dolbeault moduli space associated with the curve
and the reductive group $G$  is independent of the curve.

Theorem  \ref{ewq} concerning the Hitchin morphism is proved in \S\ref{gigi} as a corollary to the more general Theorem \ref{finit}, 
which provides  a sufficient condition for when the specialization morphism is well-defined
and is a filtered isomorphism.

$\,$
\subsection{Motivation from Gopakumar-Vafa invariants}\la{more}$\;$

One motivation for this work is to help understand a conjecture on Gopakumar-Vafa invariants
proposed in \cite{MT}, following earlier work of Kiem-Li \cite{kiem-li}.  In this section,
we briefly sketch this connection.

Let $X$ denote a Calabi-Yau threefold and let $\beta \in H_2(X,\mathbb{Z})$ denote a curve class on $X$.
There are various approaches to defining virtual counts of curves on $X$ in class $\beta$ -- Gromov-Witten or Pandharipande-Thomas invariants, for example; 
While these typically produce an infinite series of invariants, a physics proposal of Gopakumar-Vafa suggests that they should be governed 
by finitely many integers determined from the cohomology of a space of one-dimensional sheaves.

In \cite{MT}, the authors suggest a rigorous approach to their proposal using the perverse filtration as follows.
Consider the moduli space $M_\beta(X)$ of stable one-dimensional sheaves on $X$ with support cycle class $\beta$ and
Euler characteristic $1$, and let
$$\pi: M_\beta(X) \rightarrow \mathrm{Chow}_\beta(X)$$
denote the Hilbert-Chow morphism to the Chow variety of $X$ which remembers the support cycle.
\footnote{One needs to pass to seminormalizations to define this.}
Under certain assumptions, $M_\beta(X)$ carries a natural perverse sheaf $\phi_M$ associated to its shifted symplectic structure,
and one can study the perverse cohomology sheaves of its pushforward.  
More precisely, in \cite{MT}, the authors define Gopakumar-Vafa invariants $n_{g,\beta}$ of $X$ in class $\beta$ via the identity
\begin{align}
\sum_{i \in \mathbb{Z}} 
\chi(\pH^i(R \pi_{\ast}\phi_M))y^i=
\sum_{g\geq 0}n_{g, \beta}(y^{\frac{1}{2}}+y^{-\frac{1}{2}})^{2g}. 
\end{align}
Conjecturally, these GV invariants should, after a universal change of variables, agree with the Gromov-Witten invariants of $X$ in class $\beta$.  Since the precise relationship is somewhat intricate, we refer the
reader there for details.

In order for this picture to be plausible, it is necessary for GV invariants to be invariant under deformations of the Calabi-Yau threefold $X$ -- indeed such deformation invariance is built into the intersection-theoretic constructions of GW/PT invariants.  One can view the main theorem of this paper as evidence 
for this invariance; for example, one consequence of our main result is the deformation invariance of $n_{g,\beta}$ for local del Pezzo surfaces.
While there is no discussion of the sheaf $\phi_M$ in this paper, one might hope to extend the technique here to this more general setting.  

\subsection{Outline}\la{otline}$\;$

\S\ref{notation} sets up the notation. More precisely: \S\ref{gennot} sets up the notation in the context of the constructible derived category;
\S\ref{vncf} is concerned with the formalism of the vanishing/nearby cycles functors and of the specialization morphism.

\S\ref{prep} deals with the perverse Leray  filtration and specialization: \S\ref{cdte} contains Proposition
\ref{iotanoc}, which provides a  sufficient condition for when perverse truncation and (shifted) restriction to a Cartier divisor commute;
\S\ref{cosppf} contains our main technical result,  the aforementioned Theorem \ref{finit}

\S\ref{rrr} contains our main application to the Hitchin morphism: \S\ref{snaht} is a refresher on Dolbeault and Betti moduli spaces
and also proves some preliminary facts needed in the proof, given in   \S\ref{gigi},  of  Theorem \ref{ewq}, which in turn implies
Theorem \ref{ewq0} on the local constancy of the perverse Leray filtration for the Hitchin morphism
over a base.

\subsection{Acknowledgements}
We thank: Jochen Heinloth, Luca Migliorini,  J\"org Sch\"urmann, Geordie Williamson, Zhiwei Yun
for useful suggestions. The first-named author, who is partially supported by N.S.F. D.M.S. Grant n. 1600515, would like to thank the Freiburg Research Institute for Advanced Studies  for the perfect working conditions; the research leading to these results has received funding from the People Programme (Marie Curie Actions) of the European Union's Seventh Framework Programme
(FP7/2007-2013) under REA grant agreement n. [609305].  The second-named author is supported by NSF FRG grant DMS 1159265.

$\,$
\section{Notation}\la{notation}

\subsection{General notation}\la{gennot}$\;$

By variety, we mean a separated scheme of finite type over the field of complex numbers
$\comp$.  By point, we mean a closed point. See \ci{bams} for a quick introduction and for standard references for many of the
concepts and objects in this subsection.

Given a variety $Y$, we denote by $D(Y)$ the constructible bounded derived category of  sheaves of rational vector spaces on $Y,$
endowed with the middle perversity $\frak{t}$-structure.  We denote derived functors using the un-derived notation, e.g. if $f:X\to Y$
is a morphism of varieties, then the derived direct image (push-forward) functor  $Rf_*$ is denoted by $f_*$, etc.
Distinguished triangles  in $D(Y)$ are denoted $G' \to G \to G'' \rightsquigarrow$. 
The full subcategory of perverse sheaves is denoted by $P(Y)$.
We employ  the following standard notation for the objects associated with this $\frak{t}$-structure:
the full subcategories
$ ^p\!D^{\leq j}(Y)$ and ${^p\!D}^{\geq j}(Y)$, $\forall j \in \zed$, of $D(Y)$, and
${^p\!D}^{[j,k]}(Y): =    {^p\!D}^{\geq j}(Y) \cap    {^p\!D}^{\leq k}(Y),$ $\forall j\leq k \in \zed$; the truncation functors 
$\ptd{j}: D(Y) \to  {^p\!D}^{\leq j}(Y)$ and  $\ptu{j}: D(Y) \to {^p\!D}^{\geq j}(Y)$;  the perverse cohomology functors
$\ph^j: D(Y) \to P(Y)$. At times, we drop the space variable $Y$ from the notation.

%A complex (or sheaf) in $D(Y)$ is said to be lisse, if all of its cohomology sheaves are locally constant. 

The following operations preserve constructibility of complexes:
ordinary and extraordinary push-forward and pull-backs, hom and tensor product,
Verdier duality.    The nearby and vanishing cycle functors also preserves constructibility, with the provision
that, when dealing with these functors, one is working over a disk.

The $k$-th  (hyper)cohomology
groups of $Y$ with coefficients in $G \in D(Y)$ are denoted by $H^k(Y,G)$. The complex computing
this cohomology is denoted by $R\Gamma (Y,G)$ and it lives in the bounded derived category
$D({\rm point})$
whose objects are    complexes of vector spaces with cohomology given by finite dimensional rational vector spaces.

The filtrations we consider are finite and increasing. 
For every $G \in D(Y)$, the $\frak{t}$-structure defines a natural filtered object $(R\Gamma (Y,G), P)$, and $P$ is called the perverse filtration.

If a statement is valid for every value of  an index, e.g. the degree of a cohomology group, or the step of a filtration, then we denote such an index by a bullet-point symbol, or by a star symbol: $H^\bullet (X,F)$, $P_{\bullet}$, $\ptd{\bullet}$,  $P_\star H^\bullet (X,F)$, ${\rm Gr}^{P_f}_\star H^\bullet (X,F)$.

Given $G\in D(Y)$ we have the  system of truncation morphisms:
\beq\la{2wxs}
\xymatrix{
\ldots \ar[r] & \ptd{k-1} G \ar[r] & \ptd{k} G \ar[r] & \ldots \ar[r]& G
}
\eeq
A morphism $G\to G'$ in $D(Y)$ gives rise to a system of morphisms:
\beq\la{2w}
\xymatrix{
\ptd{\bullet} G \ar[r] & \ptd{\bullet} G'
}
\eeq
which are compatible with (\ref{2wxs}) is the evident manner. We say that  the system (\ref{2w}) is a  system
of  compatible  morphisms. It gives rise to a morphisms of filtered objects:
\beq\la{lmn}
\xymatrix{
\left( R\Gamma (Y, G), P) \right) \ar[r] &  \left( R\Gamma (Y, G'), P)\right) & {\rm in}\; DF({\rm pt}),
}
\eeq
where $DF(-)$ denotes the filtered derived category \ci{il}.

 Recall that $P(Y)$ is Artinian, so that the Jordan-Holder theorem holds in it.
The constituents of a non-zero perverse sheaf $G \in P(Y)$ are the isomorphisms classes of the
perverse sheaves appearing  in the unique and finite
collection of non-zero simple perverse sheaves appearing as  the quotients  in a Jordan-Holder filtration of $G.$
The  constituents of  a non-zero complex $G\in D(Y)$ are defined to be the constituents of all of its non-zero perverse
cohomology sheaves.

In general, we drop decorations (indices, parentheses, space variables, etc.) if it seems harmless in the context.

In the context of a Cartesian diagram of varieties, we denote parallel arrows by the same symbol. This should not lead to confusion in expressions like the base change morphism $g^*f_*  \to f_*g^*$.

We are going to make heavy use of the nearby/vanishing cycle functors. See \S\ref{vncf} for the basic facts.

We are going to use the decomposition theorem for semisimple complexes,  i.e. isomorphic to direct sums of shifted
simple perverse sheaves, due to Mochizuki, Sabbah and others; see the references in 
\ci{dess}.

\begin{tm}\la{dt}
Let $f:X\to Y$ be a proper morphism of varieties and let $F \in D(X)$ be semisimple. Then $f_* F$ is semisimple.
\end{tm}

$\,$
\subsection{The vanishing and nearby cycles formalism}\la{vncf}$\;$

Standard references  for this section are \ci[XIII, XIV]{sga72}  and \ci[Ch. 8,10]{kash}.

Let $S$ be a nonsingular and connected curve and consider a morphism of varieties:
\[v=v_Y: Y\to S.\]
Let $s \in S$ and let $i:s \to S$ be the closed embedding; this is what we call the special point. Objects restricted to $s$
maybe be denoted with a subscript $s$; e.g. $Y_s:= v^{-1}(s)$.

In what follows, in Choice \ref{sch}, we 
choose a disk $\disk \subseteq S,$ centered at $s.$ For the relevance of this choice to this paper, see Remark \ref{noh}.

Let $Y_\disk:= v_{Y}^{-1} (\disk)$ and let $v_{Y_\disk}:={v_Y}_{|Y_\disk}: Y_\disk \to \disk.$

We have the functors $i^*, i^!, \psi:=\psi_{Y_\disk}$ and  $\phi:=\phi_{Y_\disk}$:
\beq\la{fctz}
i^*,i^!, \psi, \phi:  D(Y) \to D(Y_s),
\eeq
where  $\psi$ is the nearby cycle functor and   $\phi$ is the vanishing cycle functor.

We have  the two, Verdier-dual, canonical distinguished triangle of functors: (we denote by the same symbol the dual arrows $\s$)
\beq\la{1p}
\xymatrix{
i^*[-1]  \ar[r]^-{\s} & \psi [-1]  \ar[r] & \phi   \ar@{~>}[r]   &, &
\phi \ar[r] & \psi[-1]  \ar[r]^-{\s} & i^![1]   \ar@{~>}[r] &.
}
\eeq

Recall that: 
\beq\la{tsh}
[\star] \circ \ptd{\bullet}= \ptd{\bullet -\star} \circ [\star], \qquad \mbox{ditto for $\ptu{\bullet}$ and $\ph^\bullet$}.
\eeq

\begin{fact}\la{psp0}

{\rm ({\bf 
$\frak{t}$-exactness for $\psi [-1]$ and $\phi$})} 
The functors $\psi [-1]$  and $\phi$ are exact functors of triangulated categories,  are $\frak t$-exact, and they commute with Verdier duality. In particular, they commute with the formation of the perverse cohomology sheaves functors $\ph^{\bullet}$ and with
the perverse truncation functors $\ptd{\bullet}$ and $\ptu{\bullet}$. We thus  have the following canonical identifications: 
\beq\la{can1}
\xymatrix{
\ptd{\bullet} \phi =  \phi \ptd{\bullet}; & \ptd{\bullet} \psi [-1] =  \psi [-1] \ptd{\bullet};&
 \mbox{ditto for $ \ptu{\bullet}$ and  $\ph^{\bullet}$}.
}
\eeq
\end{fact}

 The following is a key property of the vanishing cycle functor.
 
 \begin{fact}\la{psp}
{\rm ({\bf 
Vanishing of $\phi$ for smooth morphism and lisse  sheaves})}
If $v_{Y_\disk}: Y_\disk \to \disk$ is smooth and  $G \in D(Y)$ has locally constant cohomology sheaves    on $Y_\disk$,
then $\phi\,G=0 \in D(Y_s)$. See \ci[XIII, 2.1.5]{sga72}. Note the special case where $Y=S$ and $v_Y$ is the identity.
\end{fact}

 \begin{fact}\la{gh1}
The composition $i^*[-1] \to \psi[-1] \to i^![1]$ yields a morphism of functors $D(Y) \to D(Y_s)$:
\beq\la{gh}
i^*[-1]\lorw  i^![1]
\eeq 

The  morphism (\ref{gh})  is an isomorphism when evaluated on a complex $G\in D(Y)$ such that
$\phi \, G=0;$ in this case, we have isomorphisms: 
\beq\la{ii}
\xymatrix{
i^*[-1] G \ar[r]^-\simeq &  \psi [-1] G  \ar[r]^-\simeq  & i^![1] G
}
\eeq

 The morphism {\rm (\ref{gh})}
coincides with the morphism obtained via Verdier's specialization functor, (cf. \ci{jorg}, for example), so that it depends  only on the closed embedding $Y_s\to Y,$
i.e. it is independent of $v_Y$ and of the choice of the disk $\disk.$ 
\end{fact}

\begin{fact}\la{psphit}
({\bf 
Base change diagrams for $\psi$ and $\phi$})
Let $f:X\to Y$ be an $S$-morphism and let $v_X= v_Y \circ f:X\to S$.
The  base change morphisms associated with  $i$ and $f$ give rise to 
morphisms of distinguished triangles of morphisms as follows (cf. \ci[XIII, 2.1.7]{sga72}):

\ben

\item\la{pp3b}
\beq\la{fdtbis}
\xymatrix{
i^* [-1] f_*   \ar[r]^{\s\, \circ f_*} \ar[d] &   \psi_Y [-1]  f_*   \ar[r] \ar[d]  &   \phi_Y f_* \ar[d]  \ar@{~>}[r] & \\
f_* i^* [-1]  \ar[r]^{f_* \circ \, \s}   &  f_* \psi_X  [-1]  \ar[r]   & f_* \phi_X   \ar@{~>}[r]  &.
}
\eeq
When $f$ is proper, 
the  morphism (\ref{fdtbis}) is an  isomorphism.

\item\la{p!}

\beq\la{fd1}
\xymatrix{
 \phi_Y f_*    \ar[r]  \ar[d]  &   \psi_Y [-1] f_*    \ar[r]^-{\s \circ f_*}  \ar[d]  &  i^! [1] f_*   \ar[d]^-=    \ar@{~>}[r] &
\\
  f_* \phi_X \ar[r]      &  f_* \psi_Y[-1]     \ar[r]^-{f_* \circ \s}      & f_*i^![1]  \ar@{~>}[r]  &,
}
\eeq
When $f$ is proper, 
the  morphism (\ref{fd1}) is an  isomorphism.

\item
by combining  the  (\ref{fdtbis}) with   (\ref{fd1}),  we obtain the following commutative diagram:
\beq\la{r00}
\xymatrix{
i^*[-1] f_* \ar[d]    \ar[r]   & \psi[-1] f_* \ar[r]    \ar[d]    & i^![1] f_* \ar[d] 
\\
f_* i^*[-1]      \ar[r]   & f_* \psi[-1]  \ar[r]        &      f_* i^![1].
}
\eeq
\een

\end{fact}

Let $t \in S$ be another point and, by abuse of notation, denote the closed embedding $t\to S$ also by $t.$ 
There is the natural morphism (cf. (\ref{gh})):
\beq\la{ght}
t^*[-1]\lorw  t^![1].
\eeq
\begin{fact}\la{gogo}
Let $G\in D(Y)$ and let $t \in S$ be a  general point. Then the natural morphism $t^* [-1] G\to t^![1]G$ is an isomorphism
in $D(Y_t).$ 
To see this,   recall {\rm (\ref{gh})}, and use the vanishing of   vanishing cycle functor ``translated to $t$"  \ci[Rmk. 4.2.4]{jorg}. As usual, here ``general" means that it can be chosen to be any point of a suitable Zariski-dense 
open subset $S^o\subseteq S$ that depends on $G \in D(Y).$ 
\end{fact}

The following follows from the fact that for $t$ general, $t^*[-1], t^![1]$ commute with all the functors involved in the
constructions of perverse truncation.  
\begin{fact}\la{vbfz} 
{\rm ({\bf $\frak{t}$-exactness and $t^*, t^!$})}
Let $G \in D(Y)$ and let $t\in S$ be general. Then
we have:
\beq\la{vbgfz}
t^*[-1] \ptd{\bullet} G = \ptd{\bullet } t^*[-1] G, \quad t^![1] \ptd{\bullet} G = \ptd{\bullet } t^![1] G,
\quad \mbox{ ditto for  $\ptu{\bullet}$ and  $\ph^{\bullet}$.}
\eeq
%and the same for $\ptu{\bullet}$ and for  $\ph^{\bullet}.$

\end{fact}

\begin{fact}\la{gi9}
What follows is a consequence of Deligne's generic base change theorem \ci[Thm. 9.1]{sga4.5} 
  and of stratification theory; see also the discussion in \ci{desomm}.
Given a finite collection of morphisms $v_i: Y_i\to S$ and complexes $G_i \in D(Y_i),$ there is a Zariski-dense open subset
$S^o \subseteq S$ such that the direct images ${v_i}_* G_i$ are lisse, and such that their formation commutes with arbitrary base change.
By shrinking $S^o$ if necessary, we can further assume that 
the $G_i$ have no constituent supported on fibers of the morphisms $v_i$ over $S^o$, and that the strata on the $Y_i$, of stratifications
with respect to which the  $G_i$ are constructible, map smoothly to $S$ over $S^o$ (cf. \ci[Rmk. 4.2.4]{jorg}).
The points $t\in S^o$ are said to be general for  the finite collection $\{G_i\}.$ 
\end{fact}

\begin{defi}\la{tge}
We say that $t \in S$ is general (for the collection of $G_i$'s) if $t \in S^o= S^o(G)$ (cf. Fact {\rm \ref{gi9}}).
\end{defi}

\begin{choice}\la{sch}
Let $s\in S$.  Let $S^o\subseteq S$ be the Zariski-dense open subset of points
of the connected curve $S$
which are general with respect to  some finite collection of $G_i$'s as above.
 Let $t \in S^o.$  Choose a disk $s,t \in \disk \subseteq S$ such that $\disk^*:= \disk \setminus \{s\} \subseteq S^o.$   Chose a pointed universal covering 
$(\w{\disk^*}, \w{t}) \to (\disk^*,t).$
\end{choice}

\begin{rmk}\la{idspt}
In the special case where $v_Y:Y\to S$ is the identity on $S,$ we have that $i^*,i^!, \psi ,\phi: D(S) \to D(s)$ and that $t^*, t^!: D(S)\to D(t)$.
We have canonical identifications $D(s)=D({\rm pt})= D(t),$ where  ${\rm pt}$ is just a point, so that all three categories are 
naturally equivalent to the bounded derived category of finite dimensional  rational vector spaces. Similarly, in the filtered case:
$DF(s)=DF({\rm pt})= DF(t)$.
\end{rmk}

When $v_Y:Y \to S$ is the identity, we have the following:

\begin{fact}\la{jo} {\rm{\bf (Fundamental isomorphism)}}
Let things be as in Choice {\rm \ref{sch}}. There are the natural isomorphisms:
\beq\la{000}
\xymatrix{
t^*[-1] G \ar[r]^-\simeq & \psi_\disk [-1]  G  \ar[r]^-\simeq & t^![1] G,  &{\rm in}\; D(t)=D({\rm pt}) = D(s),
}
\eeq
where all three terms are well-defined, up to canonical isomorphism, independently of the choices, but where the isomorphisms,
which depend on the choice of $\w{t},$ are uniquely defined modulo the monodromy action of the fundamental group
$\pi_1 ({\disk}^*,t).$  
See the fundamental identity \ci[XIV, 1.3.3.1]{sga72}. Note that {\rm (\ref{000})} is Verdier self-dual.
\end{fact}

\begin{fact}\la{spmz}
{\rm ({\bf Specialization morphism})}
Consider,  in the special case where   $f$ is $v_Y$,  the morphism of distinguished triangles {\rm (\ref{fdtbis})}.  
We place ourselves in the set-up of Choice {\rm \ref{sch}}.  
One would like to specialize cohomology from the special  point  $s\in \disk,$ to the general point $t \in \disk^*.$ This is not always possible, as we now discuss. 

 Diagram {\rm (\ref{fdtbis})} yields the
functorial morphism of distinguished triangles in $D({\rm pt})$:
\beq\la{jsa}
\xymatrix{
R\Gamma (s, i^*v_*G)    \ar[d] \ar[r]^-{\s_{Y_\disk}} &  R\Gamma (s, \psi_\disk v_*G) =  R\Gamma (Y_t, t^*G)  \ar[d] \ar[r]  &  R\Gamma (s, \phi_\disk [1] v_* G) 
\ar[d]   \ar@{~>}[r]  & 
\\
R\Gamma (Y_s, i^* G)   \ar[r]^-{\s_\disk} \ar@/0pt/@{-->}[ur]^-{\rm sp?}
&   R\Gamma (Y_s, \psi_{Y_\disk}  G)   \ar[r] &  R\Gamma (Y_s, \phi_{Y_\disk} [1] G) \ar@{~>}[r]  &,
}
\eeq 
where  the  canonical identification in the middle of the first row stems from  Fact {\rm \ref{jo}} and the fact that base change is an isomorphism for general $t$ (cf. Fact {\rm \ref{gi9}}). 

Note that while this identification depends on the choice of the pre-image $\w{t} \in \w{\disk^*}$
of $t \in \disk^*,$ the resulting restriction morphism $\s_{Y_\disk}: R\Gamma (Y_\disk,G) \to R\Gamma (Y_t,G)$ depends only on $t$,
and not on the choice $\w{t}$: in fact,  monodromy acts on the target, but the domain maps  into the invariants.

Let us emphasize an important piece of diagram {\rm (\ref{jsa})}: (note that the middle term below does not change
when we shrink the disk $\disk$ centered about $s$) 

\beq\la{fk1}
\xymatrix{
R\Gamma (Y_s,i^* G) && R\Gamma (Y_\disk,G) \ar[ll]_{{\rm restriction\, to}\, s}
\ar[rr]^-{{\rm restriction\, to}\, t} &&
 R\Gamma (Y_t, t^*G).
}
\eeq

In general, i.e. without any further hypothesis ensuring that the restriction to $s$ is an isomorphism,   there is no resulting natural morphism
$R\Gamma (Y_s, G) \to R\Gamma (Y_t, G)$.
When restriction to $s$ is an isomorphism,  e.g. when $v_Y$ is proper, then we call the resulting morphism
the specialization morphism:
\beq\la{spmo}
\xymatrix{
R\Gamma (Y_s,G) \ar[r]^{\rm sp}&
 R\Gamma (Y_t,G).
}
\eeq
Of course, even if $v_Y$ is not proper, it may happen that there is a well-defined specialization morphism
for some $G \in D(Y)$.

When the specialization morphism is well-defined, we have  the distinguished triangle:
\beq\la{jsac}
\xymatrix{
 v_* i^* G \simeq i^* v_* G    \ar[r]^-{\rm sp}  &   
 \psi  v_* G \ar[r] &   \phi  v_* G \ar@{~>}[r] &.
}
\eeq

\end{fact}

\begin{fact}\la{splisse}
If $G\in D(S)$ has locally constant cohomology sheaves, then a specialization morphism is an isomorphism. This 
follows at once from Fact \ref{psp}.
\end{fact}

\begin{rmk}\la{noh}
The Choice  {\rm \ref{sch}} is harmless for our purposes: in fact, when defined, the  specialization morphism  depends only on $s$ and on $t \in S^o\setminus \{s\}$. 
Note also that the morphisms {\rm (\ref{ght})} are independent of the choice of the disk $\disk$.
\end{rmk}

$\,$
\section{Perverse filtration and specialization}\la{prep}

If one analyzes the behavior of the perverse filtration in families via the specialization morphism, there are two issues.  First, the specialization morphism is not defined in general if $X$ is not proper over $S$, which is the case for families of Dolbeault moduli spaces. The second issue is that, even when the specialization morphism is defined,  as it is in the case of the Dolbeault moduli spaces, it gives rise to a filtered morphism for the perverse Leray filtrations.  Its failure to be a filtered isomorphism is detected by the filtered cone; however, in general, this is not the filtered cone associated with the  natural morphism of functors $i^* \to \psi$ ($i^*$ the pull-back to the special fiber, $\psi$ the nearby fiber functor). This is due to the fact that perverse truncation does not commute with restriction/pull-back: e.g. when one has a direct summand supported on the special fiber.  

The goal of this section is to prove
Theorem \ref{finit} in \S\ref{cosppf}, which is a criterion to have a well-defined   specialization morphism which is a filtered isomorphism
for the corresponding perverse filtrations.  To this end, 
in \S\ref{cdte}, we study a bit the relationship between perverse truncation and restriction to a Cartier divisor.

$\,$
\subsection{Cartier divisors and  partial $\frak{t}$-exactness}\la{cdte} $\;$

The purpose of this section is to prove Proposition \ref{iotanoc}, especially equation (\ref{tuchm}).
We recommend that  readers skip this section at a first reading. 

The next lemma records some general   $\frak{t}$-exactness properties related to embeddings of effective Cartier divisors on
varieties.

\begin{lm}\la{iotatex}
Let $\iota: T' \to T$ be a closed embedding  of varieties such that the open embedding  $T\setminus T' \to T$ is an affine morphism
(e.g. $T'$ is an effective Weil divisor supporting an effective Cartier divisor). Then: (we omit
 the space variables)
 
 \ben
 \item The functor $\iota^*$   is right $\frak{t}$-exact and the functor $\iota^!$ is left $\frak{t}$-exact, i.e.:
 \beq\la{isrt}
 \xymatrix{
 \iota^*: {^p\!D}^{\leq \bullet}  \to {^p\!D}^{\leq \bullet} , &  \iota^!: {^p\!D}^{\geq \bullet}  \to {^p\!D}^{\geq \bullet}.
 }
 \eeq
 \item
The functor $\iota^*[-1] $ is left  $\frak{t}$-exact and the functor $\iota^! [1]$ is 
right $\frak{t}$-exact:
\beq\la{isrt1}
\xymatrix{
\iota^*: {^p\!D}^{\geq \bullet}  \to {^p\!D}^{\geq \bullet -1}, & \iota^!: {^p\!D}^{\leq \bullet}  \to {^p\!D}^{\leq \bullet +1}.
}
\eeq
\een
\end{lm}
\begin{proof}
The inequalities (\ref{isrt}) are  \ci[4.2.4]{bbd}),  in fact, they are valid for any immersion.  What follows is specific
to the situation of the Lemma.

We prove the inequalities (\ref{isrt1}).
 Recall that, by  \ci[4.1.10.ii]{bbd}, we have:  if $G \in P(T)$,  than  $\iota^* G \in {^p\!D}^{[-1,0]}$. 
The desired inequality  for $\iota^*$
follows from this fact and  a  simple
descending induction on $\bullet$, by  using $\iota^*$ of  the truncation distinguished  triangles ${\ph}^\bullet [-\bullet] \to \ptu{\bullet} \to \ptu{\bullet +1}
\rightsquigarrow	$, and  $\iota^*{\ph}^{\bullet +1} \in {^p\!D}^{[-1,0]}$.

Since, for $G \in P(T)$, we have that $\iota^! G \in {^p\!D}^{[0,1]}$, the proof for $\iota^!$ is analogous. Alternatively, it can  also be deduced from the one for $\iota^*$ by Verdier duality.
\end{proof}

\iffalse
\begin{rmk}\la{lmim}
 Lemma {\rm \ref{iotatex}} implies that, for every $G \in P(T),$  we have that $\iota^* \ptd{\bullet}G  \in {^p\!D}^{[-1,0]}(T')$, and that 
$\iota^! \ptd{\bullet} \in {^p\!D}^{[0,1]} (T')$.
\end{rmk}
\fi

The following lemma is a technical preliminary to Proposition \ref{iotanoc}.

\begin{lm}\la{trmeglio} {\rm ({\bf  Canonical factorization of $\ptd{\bullet -1} i^* \to \ptd{\bullet} i^*$})}
Let $\iota: T' \to T$ be as in Lemma {\rm \ref{iotatex}}.
The natural morphism $\gamma: \ptd{\bullet -1} \iota^* \to  \ptd{\bullet} \iota^*$ admits a canonical factorization:
\beq\la{nmf}
\xymatrix{
\gamma: \ptd{\bullet -1} \iota^* \ar[r]^-\delta &  \iota^* \ptd{\bullet} \ar[r]^-\e & \ptd{\bullet} \iota^*.
}
\eeq
Similarly, for the dual natural morphism   $\xymatrix{\ptu{\bullet +1} \iota^! &  \ptu{\bullet} \iota^!: \gamma' \ar[l]}$:
%(arrows reversed, $(\bullet -1)$ replaced by $(\bullet +1)$):
\beq\la{nmfd}
\xymatrix{
\ptu{\bullet +1} \iota^!  &  \iota^! \ptu{\bullet}  \ar[l]_-{\delta'}& \ptu{\bullet} \iota^! : \gamma'  \ar[l]_-{\e'} 
}
\eeq
\end{lm}
\begin{proof}
We start with the following diagram of distinguished triangles in $D(T')$: (the arrows $\gamma$ and $\delta$ are not part
of either distinguished triangle; they are there to help visualize the situation)
\beq\la{tmeq1}
\xymatrix{
&&\iota^* \ptd{\bullet} \ar[r] \ar@/^0pt/@{-->}[d]^-{\e} & \iota^* \ar[r]  \ar[d]^-=& \iota^* \ptu{\bullet +1} \ar@/^0pt/@{-->}[d] \ar@{~>}[r] &
\\
(\ptd{\bullet -1} \iota^* \ar[r]^-\gamma \ar@/^0pt/@{-->}[rru]^-{\delta} &) & \ptd{\bullet} \iota^*   \ar[r]   &  \iota^*    \ar[r]  &  \ptu{\bullet +1} \iota^*  \ar@{~>}[r] &.
}
\eeq

By the l.h.s. of  (\ref{isrt1}), we have that:
\beq\la{rr}
\iota^* \ptu{\bullet +1}: D(T) \to {^p\!D}^{\geq \bullet}(T').
\eeq
 
By \ci[Prop. 1.1.9, p.23]{bbd}, the diagram (\ref{tmeq1}) can then  be completed uniquely to a morphism
of distinguished triangles; this is visualized by means of the dotted arrows.

The conclusion follows from the  inequality (\ref{rr}), by taking  the long exact sequence associated with
${\rm Hom}_{D(T')} (\ptd{\bullet-1} \iota^*, -)$ applied to the distinguished triangle on the top of
(\ref{tmeq1}):  in fact, for every $G\in D(Y)$, we have
${\rm Hom}_{D(T')} (\ptd{\bullet -1} \iota^*G, -)=0$, when evaluated at something in ${^p\!D}^{\geq \bullet}(T')$.

The proof for $\iota^!$ is analogous.  Alternatively, it can  also be deduced from the one for $\iota^*$ by Verdier duality.
\end{proof}

\begin{rmk}\la{ecz} In view of  the l.h.s. inequality in (\ref{isrt1}),   the r.h.s. vertex in (\ref{tmeq1})   satisfies the inequality
$\iota^* \ptu{\bullet +1}: D \to {^p\!D}^{\geq \bullet}$. 
By taking the long exact sequence of perverse cohomology of  the top distinguished triangle in (\ref{tmeq1}), the aforementioned inequality yields the
natural isomorphism of functors: 
\beq\la{ecz1}
\xymatrix{
\ptd{\bullet -1} \iota^* \ptd{\bullet} \ar[r]^-\simeq & \ptd{\bullet-1} \iota^*.
}
\eeq
\end{rmk}

\begin{pr}\la{iotanoc} {\rm ({\bf No constituents on divisors and $\frak{t}$-exactness})}
Let $\iota: T' \to T$  be as in Lemma {\rm \ref{iotatex}}. 
If $G \in D(T)$ has no constituent supported on $T'$, then the morphisms $\delta$  (cfr. {\rm (\ref{nmf})}) are isomorphisms,
and we get natural isomorphisms: 
\beq\la{tuch}
\xymatrix{
\ptd{\bullet -1} \iota^* G \ar[r]^-\delta_-\simeq &\iota^* \ptd{\bullet} G , & \ptu{\bullet -1} \iota^* G  \ar[r]^-{\simeq} & \iota^* \ptu{\bullet} G 
  , &
{\ph^{\bullet -1}} \iota^* G  \ar[r]^-\simeq & 
\iota^* [-1] { \ph^{\bullet}} G.
}
\eeq
The same holds if we replace $\iota^*$ with $\iota^!$ and $(\bullet -1)$ with $(\bullet +1)$.

Equivalently, we have:
\beq\la{tuchm}
\xymatrix{
\ptd{\bullet} \iota^* [-1]  G \ar[r]^-\delta_-\simeq &\iota^*[-1] \ptd{\bullet}  G , \qquad \ptd{\bullet} \iota^! [1]  G
 & \iota^![1] \ptd{\bullet}  G   \ar[l]_-\delta^-\simeq , &
 {\ph^{\bullet +1}} \iota^! G   & \iota^! [1] { \ph^{\bullet}} G   \ar[l]_-\simeq.
}
\eeq

\end{pr}
\begin{proof}
It is enough to prove (\ref{tuch}): the statement for $\iota^!$ follows by Verdier duality; The equivalent statements are mere
reformulations by means of (\ref{can1}).

It is enough to prove the first statement on the l.h.s. of (\ref{tuch}), for the remaining ones follow formally by consideration of the truncation distinguished triangles.

We have that  $\ptd{\bullet} G \in {^p\!D}^{\leq \bullet}$, so that, by (\ref{isrt}), we have that  $\iota^*\ptd{\bullet} G \in {^p\!D}^{\leq \bullet}$, and then, clearly, we have that
 \beq\la{0129}
 \iota^*\ptd{\bullet} G= \ptd{\bullet} \iota^*\ptd{\bullet} G.
 \eeq
 
 In view of (\ref{0129}) and of (\ref{ecz1}), and by considering the truncation triangle
 $\ptd{\bullet -1} \to \ptd{\bullet} \to {\ph^\bullet}[-\bullet] \rightsquigarrow$ applied to $\iota^*\ptd{\bullet} G$,
 in order to prove the first equality on the lhs in (\ref{tuch}), it is necessary and sufficient  to show that
 $\ph^{\bullet}(\iota^* \ptd{\bullet} G) =0$. 
 
 This can be argued as follows. By taking the long exact sequence of perverse cohomology of the  distinguished triangle
 $\iota^* \ptd{\bullet -1}  G \to \iota^* \ptd{\bullet} G \to \iota^* {\ph^{\bullet}} G [-\bullet]  \rightsquigarrow$, we see
 that it is necessary and sufficient to  show that $\iota^* {\ph^{\bullet}} G [-\bullet] \in {^p\!D}^{\leq \bullet -1}$, or, equivalently, that
 $\iota^* {\ph^{\bullet}} G  \in {^p\!D}^{\leq -1}.$
 
By \ci[4.1.10.ii]{bbd}, we have the distinguished triangle  ${\ph^{-1}} (\iota^* {\ph^{\bullet}} G) [1]
\to \iota^* {\ph^{\bullet}} G \to {\ph^{0}} (\iota^* {\ph^{\bullet}} G)   \rightsquigarrow$ and an epimorphism
$\ph^{\bullet} G \to {\ph^{0}} (\iota^* {\ph^{\bullet}} G)$. Since $G$  is assumed to not have constituents supported at
$Y_s$, we must have $\ph^{0} (\iota^* {\ph^{\bullet}} G)=0$, so that $\iota^* {\ph^{\bullet}} G [-\bullet] \in {^p\!D}^{\leq \bullet -1}$,
as requested. The l.h.s. equality in (\ref{tuch}) follows, and we are done.
  \iffalse
 For the equality in the middle of (\ref{tuch}), we argue as follows. Since $\iota^* [-1]$ is left $\frak{t}$-exact, we have  that $\iota^*
 {\cre \ptu{\bullet}} G \in {^p\!D}^{\geq \bullet -1}$ (cf. \ref{isrt1}).
 
 Consider the distinguished triangle
 $\iota^*  \ptd{\bullet -1} G \to \iota^*G \to \iota^* \ptu{\bullet}  G \rightsquigarrow$.  By the l.h.s. equality in (\ref{tuch}), 
 the distinguished triangle directly above reads
$\ptd{i-2}  \iota^* G \to \iota^* G \to  \iota^* \ptu{\bullet} G  \rightsquigarrow$  (where the first arrow is the natural one). This forces
the second equality in (\ref{tuch}). If one prefers not to verify the compatibility in the parentheses, observe that:
since $\iota^* [-1]$ is left $\frak{t}$-exact, we have  that $\iota^* \ptu{\bullet} G \in {^p\!D}^{\geq \bullet -1}$; this forces the
distinguished triangle above to be the truncation distinguished triangle and the middle equality  in (\ref{tuch})  
follows again.

The r.h.s. equality in (\ref{tuch})  for the perverse cohomology sheaves follows formally from the ones we have established for the
truncations. This concludes the proof of the statements related to $\iota^*$.

The proof for $\iota^!$ is analogous. Alternatively, it can  also be deduced from the one for $\iota^*$ by Verdier duality.
\fi
\end{proof}

We shall use the following simple lemma, where we employ the notation in \S\ref{vncf}.

\begin{lm}\la{silm}
Let $G \in D(Y)$ be such that $\phi_{Y_\disk} G=0.$ Then no constituent of $G$ is supported on $Y_s.$
\end{lm}
\begin{proof}
By chasing the definitions, we see that we may assume that $G$ is perverse. By using the fact that $\phi$ is $\frak{t}$-exact, and by an easy induction on the length of a Jordan-H\"older filtration
of $G,$ we see that $\phi G'=0$ for every constituent $G'$ of $G.$ The conclusion follows from the fact that if $G' \in D(Y)$ is a non-zero complex supported 
on $Y_s$, then $i_*\phi G'=G'.$
\end{proof}

 By combining Lemma \ref{silm} with  Proposition \ref{iotanoc}, we obtain the following

\begin{cor}\la{dodo}
Let $G \in D(Y)$ be such that $\phi_{Y_\disk} G=0.$ Then we have a natural isomorphism:
\beq\la{zzx}
\delta :\ptd{\bullet} i^*[-1] G \stackrel{\simeq}\to
i^*[-1] \ptd{\bullet} G; \quad  \mbox{similarly, for $\ptu{\bullet}$, and for $\ph^{\bullet}$; ditto, for $i^![1]$.}
\eeq
\end{cor}

$\,$
\subsection{Compatibility of the specialization morphism with  the perverse filtration}\la{cosppf}$\;$

\begin{tm}\la{finit}
Let $f:X\to Y$ be proper morphism, let  $v_Y: Y\to S$ be a  morphisms onto a nonsingular connected curve and let $F \in D(X)$. 
Choose $s,t \in  S$ and a disk $\disk$ as in Choice \ref{sch}.  Assume one of the following conditions:
\ben
\item[(i)]
$\phi  F=0$ and $v_Y$ is proper, or

\item[(ii)] $\phi F=0$ and the $v_* \ptd{\bullet} f_* F$ have locally constant cohomology sheaves on $S$, or

\item[(iii)] $F$ is semisimple, $\phi F=0$ and  $v_* f_* F$ has locally constant cohomology sheaves on $S$.
\een

Then the specialization morphism is well-defined and it is a filtered isomorphism for the respective perverse Leray filtrations:
\beq\la{5050}
 \xymatrix{
 {\rm sp}: 
 \left(R\Gamma (X_s, i^* F), P^{f_s}\right) \ar[r]^-\simeq & \left(R\Gamma (X_t, t^* F), P^{f_t}\right).  
 }
 \eeq
 \end{tm}

\begin{proof}
By applying (\ref{r00}) to the $\ptd{\bullet} f_*$, 
we obtain the following commutative diagrams: 
\beq\la{arce}
\xymatrix{
i^*[-1] v_* \ptd{\bullet} f_* \ar[r]^-1  \ar[d]^-3 & \psi [-1] v_* \ptd{\bullet} f_*  \ar[d]^-4   \ar[r]^-2  & i^![1] v_* \ptd{\bullet} f_* \ar[d]^-=_-5
\\
 v_* i^*[-1] \ptd{\bullet} f_* \ar[r]^-{1'}  \ar@/0pt/@{-->}[ur]^-{\rm sp?}  &  v_*  \psi [-1]  \ptd{\bullet} f_*    \ar[r]^{2'}  &  v_*   i^![1] \ptd{\bullet} f_* 
}
\eeq
Up to shift:  the cones of $1$ and $2$ coincide with $\phi v_* \ptd{\bullet} f_*$; the cones of $1'$ and $2'$ coincide with 
$v_*\phi \ptd{\bullet} f_*$.

Let us prove (i). 

In this case, we only need the commutative square on the l.h.s. of (\ref{arce}).

Since $v$ is assumed to be proper, the base change morphisms $3$ and $4$ are isomorphisms. In particular,
the specialization morphism  ${\rm sp}$ is well-defined and it gives rise to a   system of compatible morphisms.

The assumption $\phi \, F=0$, which  is common to (i) and (ii),  implies,  by the $\frak t$-exactness of $\phi$ and the properness of $f$, that 
  $v_*\phi \ptd{\bullet} f_*F=  v_* \ptd{\bullet} f_* \phi F =0$, i.e. that  the cone of $1'$ is zero, so that so is the cone of $1$.
  
  It follows that (i) is a  system of compatible isomorphisms and, as such, it gives rise to an isomorphism
  in the filtered derived category.
  
  The filtered complex $R\Gamma (\mo_s, i^*[-1] F), P^{f_s})$ arises in connection with the cohomology of the compatible system of complexes
$v_* \ptd{\bullet} f_* i^*[-1] F)$. Similarly, the filtered complex $R\Gamma (\mo_t, t^*[-1] F), P^{f_t})$ arises in connection with the cohomology of the compatible system of complexes
$v_* \ptd{\bullet} f_* t^*[-1] F)$.

It remains to identify:
\ben
\item[(a)]
 $v_* i^*[-1]  \ptd{\bullet} f_* F$ with $v_*  \ptd{\bullet} f_* [-1] i^*F$, and
 \item[(b)]
 $\psi [-1] v_*   \ptd{\bullet} f_* F$ with $v_*  \ptd{\bullet} f_* t^* [-1] F$.
 \een

To achieve (b), we argue as follows.
The choice of $t$ general for  $F,$ made in Definition
\ref{tge}, allows us to:  
replace $\psi $ with $t^*$ (cf. (\ref{000})); use the identification $t^*v_*=v_* t^*$; use the identification 
$t^*[-1] \ptd{\bullet} = \ptd{\bullet} t^*[-1]$ (cf. \ref{vbgfz}));  use the identification $t^*v_*=v_* t^*$. Then (b) follows.

To achieve (a), we argue as follows. We first apply Corollary \ref{dodo} to $G=f_* F$; the condition $\phi f_* F=0$ is met
in view of the properness of $f$ and of  $\phi F=0$. We then apply proper base change $i^* f_* \stackrel{\simeq}\to f_* i^*$.  
Then (a) follows, and (i) is proved.

Let us prove (ii).

As it has been seen above, the assumption $\phi F=0$ implies that the cones of  $1'$ and $2'$ are zero, so that $1'$ and $2'$
are isomorphisms.

Since we are assuming that the $v_* \ptd{\bullet} f_* F$ have locally constant cohomology sheaves on $\disk$, we have that
the cones of $1$ and $2$ are zero as well (cf. Fact \ref{psp}), so that $1$ and $2$ are isomorphisms.

Since the morphism $5$ is an isomorphism, all the morphisms in (\ref{arce}) are isomorphisms.

We conclude as in case (i).

Case (iii) is weaker than case (ii); we can also prove it without resorting to the use of Corollary \ref{dodo}. The proof is very similar, except 
that in order to achieve the critical commutation property {\rm (a)} via Corollary \ref{dodo},  we use that: 
$F$ semisimple implies $f_*F$ semisimple (cf. the decomposition Theorem \ref{dt}); the assumption $\phi F=0$ implies that no simple summand
of $f_* F$ is supported on $Y_s$; the commutation property for a simple, un-shifted, simple perverse summand $P$, which  we know not to be  supported
on $Y_s$ follows from \ci[4.10.1]{bbd}, to the effect that $i^*[-1]P$ is perverse.

\end{proof}

$\,$
\section{The Hitchin morphism and specialization}\la{rrr}$\;$

\subsection{The Betti and Dolbeault moduli spaces: the $P=W$ conjecture}\la{snaht}$\;$

Let $X\to S$ be a smooth  projective morphism  over a  variety $S$ and let $G$ be a reductive group.

  The Betti moduli space $M_B(\mo/S,G)$ (cf. \ci[pp.12-15]{si2})  is a complex analytic space over $S$. The fiber over a point 
$s\in S$ is the moduli space (a.k.a. the character variety)  $M_B(\mo_s,G)$ of representations of the fundamental group of $\mo_s$ into $G.$ 

The Dolbeault moduli space $M_D(\mo/S,G)$ (cf. \ci[\S9, esp. Prop. 9.7]{si2})
is  quasi-projective over $S$; in general, it is not proper over $S$. The fiber over a point
$s\in S$ is the moduli space  $M_D(\mo_s,G)$ of principal  Higgs bundles of semiharmonic type on $\mo_s$ for the group $G$.
 
\begin{fact}\la{naht}
{\rm (\bf Non-Abelian Hodge Theorem)} There is a natural $S$-homeomorphism of the underlying topological spaces: 
(cf. \ci[Thm. 9.11, and Lm. 9.14]{si2})
\beq\la{enaht}
\xymatrix{
\Psi (\mo/S,G): M_B (\mo/S,G) \ar[r]^-\simeq & M_D (\mo/S,G).
}
\eeq

\end{fact}

To fix ideas, in what follows, we tacitly assume that $S$ and the fibers of the family $\mo/S$ are connected; such an assumption
is for ease of exposition only; see \ci[pp. 14-15]{si2}.

Choose any point $s_o\in S$.
The structural morphism $\pi_B (\mo/S,G) : M_B (\mo/S,G) \to S$ is analytically locally trivial over $S$, with transition functions with values in the
group of  $\comp$-scheme
automorphisms of the fiber $M_B(\mo_{s_o},G)$; see \ci[Lm 6.2, p.13]{si2}. More precisely: let $(\w{S}, \w{s_o}) \to (S, s_o)$ be a pointed universal covering space with associated identification of the deck transformation group with the fundamental group
$\pi_1 (S, s_o)$;   the fundamental group acts on $M_{s_o}$ via $\comp$-scheme automorphisms;
$M_B(\mo/S,G)$
is constructed as the quotient of $M_B(\mo_{s_o},G) \times \w{S}$ under the usual action of the fundamental group $\pi_1(S, s_o)$.

\begin{fact}\la{dltr}
{\rm (\bf Local triviality of the Dolbeault moduli space over the base)}
The local triviality of the Betti moduli space over the base, coupled with the Non-Abelian Hodge Theorem 
$S$-homeomorphism $\Psi$ {\rm (\ref{enaht})}, implies that the structural morphism $\pi_D (\mo/S,G) : M_D(\mo/S,G) \to S$ is topologically  locally trivial
over the base $S$.
\end{fact}

Recall that, for irreducible varieties,  the intersection cohomology complexes/groups
are  topological invariants, independent of the stratification (cf. \ci[\S4.1]{goma2}).
Note that \ci{goma2} deals with irreducible analytic varieties; 
on the other hand,  as the forthcoming  Lemma \ref{ictop} shows, if we define the intersection complex of a variety as the direct sum
of the intersection complex of its  irreducible components, then the topological invariance of the intersection cohomology complexes/groups is still valid.

In particular, given a topologically locally trivial fibration with fibers varieties, the intersection cohomology
groups of the fibers give rise to locally constant sheaves on the base.

We thank G. Williamson for suggesting the definition of the set $X'$ in the proof of the following lemma. Our original 
set $X'$ was defined using the notion of local irreducibility and lead to a more cumbersome proof.

\begin{lm}\la{ictop}
{\rm ({\bf Topological invariance of   intersection cohomology for reducible varieties})}
Let $X$ and $Y$ be complex analytic set and let $g:X\simeq Y$ be a homeomorphism of the underlying topological spaces
endowed with the classical topology. Then:

\ben
\item
The homeomorphism $g$ induces a natural bijection $\gamma: I \simeq B $ on the sets of irreducible components 
of $X$ and $Y$ such that $g$ induces homeomorphisms $X_i \simeq Y_{\gamma (i)}$, for every $i \in I$.
\item
Define the intersection complex  $IC_X$ of a complex analytic set $X$ to be the direct product of the intersection  cohomology complexes
of the irreducible components of $X$. Then the homeomorphism $u$ induces a natural  isomorphism
$IC_X = u^* IC_Y$.
\een
\end{lm}
\begin{proof}
The case when $X$ and $Y$ are  irreducible is proved by M. Goresky and R. MacPherson in 
\ci{goma2}. It is thus clear that it is enough to prove the first statement of the lemma, which is what we do next, leaving some 
elementary details to the reader.

Let $X'$ ($Y'$, resp.)  the be open subset of those points of $X$  ($Y$, resp.) admitting a classical open neighborhood all of whose points
admit a classical open neighborhood homeomorphic to a  Euclidean space of some dimension. 

Since $X'$ is defined topologically, it is clear that $g(X')=Y'$.

We have the following
inclusions of classical open subsets:
\[
X^{\rm sm} \subseteq X' \subseteq X^o \subseteq X, \quad X^{\rm sm}_i \subseteq X'_i \subseteq X^o_i \subseteq X_i, 
\]
where $X^{\rm sm}$ is the set of smooth points of the complex analytic set $X$, and $X^o$ is the complement of the union of all intersections $X_i \cap X_j$, $i,j \in I$, $i\neq j$, and where $-_i$ denotes intersection with the irreducible component $X_i$.

The classical open subsets $X^{\rm sm}_i$ and $X^o_i$ are also Zariski open,  irreducible and connected.
The classical open subset $X'_i$ is connected. The $X'_i$ are the connected components of $X'$. The homeomorphism
$g$ must respect the decomposition of $X'$ and $Y'$ into their connected components, so that we obtain the desired  bijection $\gamma: I \simeq B.$
The classical closure of  $X'_i$ is $X_i$ and since $g$ is a homeomorphism, we must have that $g$ induces a homeomorphism
$X_i \cong Y_{\gamma (i)}$.
\end{proof}

Lemma  \ref{ictop} implies at once the following
 
\begin{cor}\la{iclt}
Let $F,S$ be a varieties, let  $X$ be an $S$-variety,  and let $\tau: X \to F\times S$ be an $S$-homeomorphism. Then, for every $s \in S$, we have canonical
isomorphisms:
\[
(\tau^* pr_F^* IC_F)_{|X_s} = IC_{X_s}
\]
\end{cor}

\begin{rmk}\la{whygood}
The definition of intersection cohomology complex for reducible varieties stemming from Lemma \ref{ictop} is reasonable in view of the fact that it satisfies virtually all the usually properties of the usual intersection cohomology complex for irreducible varieties, e.g. purity,
mixed and pure Hodge structures,
Artin vanishing, Lefschetz theorems, relative hard Lefschetz, Hodge Riemann relations, decomposition theorem.  See  
\ci[\S4.6]{decpf2}, for example.
\end{rmk}

The intersection cohomology groups of the  fibers of  $\pi_B(\mo/S,G)$ give rise to locally constant sheaves $I\!H^\bullet_B (\mo/S,G)$
on $S$. In view of the homeomorphism $\Psi (\mo/S,G)$, and of the topological invariance
of intersection cohomology, 
the same applies to the Dolbeault picture, and we get locally constant sheaves $I\!H^\bullet_D (\mo/S,G)$
on $S$.

\begin{fact}\la{wlss} {\rm (\bf The Betti weight filtration is locally constant)}
By the local triviality of $\pi_B (\mo/S,G)$ over $S$, and since the transition automorphisms are compatible with mixed Hodge structures in  (intersection) cohomology,
we see  that  the weight filtration $W$ for the mixed Hodge structure for the intersection cohomology of the fibers
of $\pi_B(\mo/S,G)$ gives rise to locally constant  subsheaves $W_{\star,B} (\mo/S,G) \subseteq I\!H^\bullet_B (\mo/S,G) $ on $S$.
\end{fact}

The Dolbeault moduli space is endowed with a natural $\Gm$-action (cf. \ci[p.62, and p.17]{si2}), given by scalar multiplication on the
Higgs field.

The reference in this paragraph is  \ci[p.20-23]{si2}, which deals with the case of $G=GL_n$,  suitably adapted to
an arbitrary reductive $G$.
The Dolbeault moduli space admits the Hitchin $S$-morphism 
\beq\la{hmo}
\xymatrix{
h(\mo/S,G): M_D(\mo/S, G) \ar[r] &  \m{V}(\mo/S,G). 
}
\eeq
Here,  $\m{V}(\mo/S,G)$ is the quasi
projective $S$-variety representing the functor $(S'\to S) \mapsto \oplus_{i=1}^{{\rm rk}G} H^0
(X'/S', {\rm Sym}^{e_i(G)} \Omega^1_{X'/S'})$, where the positive integers $e_i$ are the degrees of the generators
of conjugation-invariant polynomials on the Lie algebra of $G$. The Hitchin morphism assigns to a $G$-principal Higgs bundle the symmetric polynomials
appearing as the coefficients of the  ``characteristic polynomial" of the Higgs field, viewed as sections of the appropriate sheaf. 
The Hitchin morphism is proper over $S$, hence projective over $S$  (the Dolbeault moduli space is 
quasi projective over $S$).
Domain and target  of the Hitchin morphism are endowed with a natural  $\Gm$-action (cf. \ci[p.62]{si2}), which covers the trivial action on $S$.
The $\Gm$-action on the target is contracting. 
The Hitchin morphism is $\Gm$-equivariant. 

We observe that the  mixed Hodge structure  on the intersection cohomology groups of each Dolbeault moduli space is pure
(the starting point is the $\Gm$-equivariance and the contracting action; then one can imitate the proof of     \ci[Lemma 6.1.1 and references therein, and the proof of Thm 2.4.1]{dehali}).
We do not need this fact, except to point out that  it is in sharp contrast with the  expected (known in some cases) non purity of the corresponding intersection cohomology groups of each Betti moduli space.

\begin{??}\la{4r4r}
{\rm ({\bf  Is $P=W$ in the non-Abelian Hodge theory in arbitrary dimensions?})}
Let $X$ be a connected  smooth projective variety. 
Then for the weight filtration $W_B (X,G)$, do we have  $W_{B,2\star +1} (X,G) = W_{2 \star} (X.G) \subseteq I\!H^{\bullet} (M_B(X,G))$? 
Via the  Non Abelian Hodge Theorem isomorphism $\Psi_*$, do we have that
$\Psi_* W_{B,2\star} (X,G) =  P^h_{D,\star}(X,G)$, where $P^h_D(X,G)$ is the perverse Leray filtration for the Hitchin morphism (suitably normalized, so that it  ``starts at zero")? 
\end{??}

\begin{rmk}\la{p=wdd} Actually, the $P=W$ conjecture, which is due to M.A. de Cataldo, T. Hausel and L. Migliorini,
is concerned with a twisted version of the Betti/Dolbeault moduli spaces for curves of genus $g \geq 2$.
The paper \ci{dhm} proves the validity of the conjecture in this  twisted case  
when $X$ is a curve and $G=GL_2, SL_2$ and $PGL_2$. In this twisted case over a curve of genus $g \geq 2$, the moduli spaces for  $G=GL_n, SL_n$ are nonsingular, and for   $G=PGL_n$ they are orbifolds;  Theorem \ref{ewq} applies to this situation.
\end{rmk}

In the next section, we prove Theorem \ref{ewq}, to the effect that   the perverse filtration gives rise to locally constant subsheaves  subsheaves $P^h_{D,\star}  (\mo/S,G) \subseteq I\!H^\bullet (\mo/S,G)$ on $S$.

We need the following

\begin{lm}\la{vaicdm}
Let things be as in Theorem \ref{ewq}.
Let $F$ be the intersection complex of the Dolbeault moduli space  $M_D (\mo/S,G)$. 
Then
\beq\la{zzz}
\phi F =0,
\eeq
the  specialization morphism is defined and it is an isomorphism:
\beq\la{zz}
\xymatrix{
{\rm sp}: R\Gamma (\mo_s, i^* F) \ar[r]^-\simeq  &  R\Gamma (\mo_t, t^* F).
}
\eeq
\end{lm}
\begin{proof}
The Betti  moduli space is topologically locally trivial over any disk contained in $S$. By the Non Abelian Hodge Theorem,
the Dolbeault moduli space is also  topologically locally trivial over our disk  $\disk$. Let $M_D(\mo_s,G) \times \disk \cong M_D(X_\disk/\disk,G)$ be any topological  trivialization. Then the intersection complex on $M_D(X_\disk/\disk,G)$ is the pull-back
of the intersection complex of $M_D(\mo_s,G)$ via the first projection associated with the chosen trivialization. Both the conclusions of the lemma follow.

\end{proof}

$\,$
\subsection{Proof of Theorem \ref{ewq} on the Hitchin morphism and specialization}\la{gigi}$\;$

\begin{proof}
We denote  the Hitchin morphism  (\ref{hmo})  simply by  $h:M\to \m{V}$,  we denote the structural $S$-morphisms
by $\pi:M \to S$ and $\rho: \m{V} \to S$.
Let $\m{IC}_M$ be the intersection complex of the Dolbeault moduli space $M$. 

Our first goal is to verify that we are now in the situation of Theorem \ref{finit}.(ii).

We set $(X,Y,S, f, v,F)$ to be $(M,\m{V}, S, h, \rho, \m{IC}_M)$.

The requirement $\phi F=0$ is met by Lemma \ref{vaicdm}.

As in the proof of Lemma \ref{vaicdm}, the intersection complex $\m{IC}_M$ is, locally over the nonsingular $S$, the pull-back from 
of the intersection complex of a fiber. It follows that
the direct image sheaves
$R^\bullet \pi_* \m{IC}_M = I\!H^\bullet_D(\mo/S,G)$
are locally constant on $S$, with stalks the intersection cohomology groups of the fibers of $\pi:M\to S$.

By the decomposition theorem \ci{bbd}, applied to the projective $h$ and the simple perverse sheaf   $\m{IC}_M$,
we have that the truncated $\ptd{\bullet} h_* \m{IC}_M$ are direct summands of the direct image $h_*F$.
It follows that  $\rho_* \ptd{\bullet}  h_* \m{IC}_M$ are direct summands of $\pi_* \m{IC}_M= \rho_* h_* \m{IC}_M$ in $D(S)$.

By combining the two last paragraphs, we have that the $\rho_* \ptd{\bullet}  h_* \m{IC}_M$ have locally constant cohomology sheaves.

We can thus apply Theorem \ref{finit}.(ii) (or its weaker variant (iii)), the conclusion of which is that the specialization morphism
$R\Gamma (M_s, i^* \m{IC}_M) \to R\Gamma (M_t, t^* \m{IC}_M)$ is defined and it is  a filtered isomorphism 
for the respective perverse Leray filtrations. 

Since, as it has been observed above,   the intersection complex
of $M$ restricts to the intersection complexes of the fibers $M_s$ and $M_t$, the 
 first assertion of Theorem \ref{ewq} follows.

The second assertion, i.e. the independence of $s \in S$ on a connected $S$, follows easily: pick any two points
$s,s'$ and a suitably general point $t \in S$. Then apply what we have proved to the pairs $(s,t)$ and $(s',t)$.
\end{proof}


\begin{thebibliography}{10}



\bibitem[Be-Be-De-1982]{bbd} {\em Faisceaux pervers,}  Analysis and topology on singular spaces, I (Luminy, 1981), 5-171, Ast\'erisque, 100, Soc. Math. France, Paris, 1982. 


%\bibitem[Be-Dr]{bedr}
	%A. Beilinson, V. Drinfel'd,
	%{\em Quantization of Hitchin's integrable system and Hecke eigensheaves,}
	%Preprint, available under {\tt http://www.ma.utexas.edu/benzvi/}

%\bibitem[Be-2017]{be} D. Bergh,  ``Functorial destackification of tame stacks with
%abelian stabilisers,"
%Compositio Math. 153 (2017), 1257-1315.


%\bibitem[Ch-La-2010]{chla} P.-H. Chaudouard, G. Laumon, ``Le lemme fondamental pond\'er\'e. I. Constructions g\'eom\'etriques,"
%Compositio Math. 146 (2010), 1416-1506.


\bibitem[de-2017]{dess} M.A. de Cataldo,  ``Decomposition theorem for semi-simples,"
J. Singul. 14 (2016), 194-197. 

\bibitem[d-2012]{decpf2} 
M.A. de Cataldo,
``The perverse filtration and the Lefschetz hyperplane theorem," II. J. Algebraic Geom. 21 (2012), no. 2, 305-345. 

\bibitem[d-M-2009]{bams} M.A. de Cataldo, L. Migliorini, ``The decomposition theorem, perverse sheaves and the topology of algebraic maps,"
Bull. Amer. Math. Soc. (N.S.) 46 (2009), no. 4, 535-633.

\bibitem[d-M-2009b]{desomm} M.A. de Cataldo, 
``The standard filtration on cohomology with compact supports with an appendix on the base change map and the Lefschetz hyperplane theorem," {\em Interactions of classical and numerical algebraic geometry}, 199-220, Contemp. Math., 496, Amer. Math. Soc., Providence, RI, 2009.

\bibitem[de-Hai-Li-2017]{dehali} M.A. de Cataldo, T. Haines, L. Li, ``Frobenius semisimplicity for convolution morphisms,"
Math. Z. 289 (2018), no. 1-2, 119--169.


\bibitem[de-Hau-Mi-2012]{dhm} M.A. de Cataldo, T. Hausel, L. Migliorini, ``Topology of Hitchin systems and Hodge theory of character varieties: the case $A_1$," Ann. of Math. (2) 175 (2012), no. 3, 1329-1407.


\bibitem[De-1976]{sga4.5} P. Deligne et. al., SGA 4$\frac{1}{2}$, {\em Cohomologie Etale,} Lecture Notes in Math. 569, Springer-Verlag, Heidelberg (1976). 

\bibitem[De-1972]{sga72} A. Grothendieck et. al, SGA 7.I, {\em Groupes de Monodromie en G\'eom\'trie
Alg\'ebrique,} Lecture Notes in Math. 288, Springer-Verlag, Heidelberg (1972). 


%\bibitem[Ed-Gr-1998]{edgr} D. Edidin, W. Graham, ``Algebraic cuts,"  Proc. Amer. Math. Soc. 126 (1998), no. 3, 677-685.


\bibitem[Go-Ma-1983]{goma2} M. Goresky, R. MacPherson, ``Intersection homology II," Invent. Math. 71, 77-129, 1983.

%\bibitem[Ha-1998]{ha} T. Hausel, ``Compactification of moduli of Higgs bundles," J. Reine Angew. Math. 503 (1998), 169-192.

\bibitem[Il-1971]{il} L. Illusie, {\em Complexe cotangent et d�formations, I,} 
(French) Lecture Notes in Mathematics, Vol. 239. Springer-Verlag, Berlin-New York, 1971. xv+355 pp.

\bibitem[Ka-Sh-1990]{kash} M. Kashiwara, P. Schapira, {\em Sheaves on manifolds. With a chapter in French by Christian Houzel,} Grundlehren der Mathematischen Wissenschaften, 292. Springer-Verlag, Berlin, 1990. x+512 pp.

\bibitem[KL]{kiem-li} Y. Kiem, J. Li, {\em Categorification of Donaldson-Thomas invariants via perverse sheaves}, arxiv:1212.6444.

\bibitem[MT]{MT} D. Maulik, Y. Toda, {\em Gopakumar-Vafa invariants via vanishing cycles}, Invent. Math., to appear.

%\bibitem[Mi-Sh-2018]{mish}  L. Migliorini, V.  Shende, 
%``Higher discriminants and the topology of algebraic maps," Algebr. Geom. 5 (2018), no. 1, 114-130.

%\bibitem[Mo-2007]{mo} T. Mochizuki, {\em Asymptotic behaviour of tame harmonic bundles and an application to pure twistor D-modules, I,} Mem. Amer. Math. Soc. 185 (2007), no. 869, xii+324 pp.

%\bibitem[Sa-2005]{sa} C. Sabbah, {\em Polarizable twistor D-modules}, Asterisque No. 300 (2005), vi+208 pp.

%\bibitem[Sh-1998]{schmitt} A. Schmitt, ``Projective moduli for Hitchin pairs," Internat. J. Math. 9 (1998), no. 1, 107-118.

\bibitem[Sc-2003]{jorg} J. Sch\"urmann, {\em Topology of singular spaces and constructible sheaves,} Instytut Matematyczny Polskiej Akademii Nauk. Monografie Matematyczne (New Series), 63. Birkhauser Verlag, Basel, 2003. x+452 pp.

\bibitem[Si-1994]{si2} C. Simpson, ``Moduli of representations of the fundamental group of smooth projective varieties II," 
Inst. Hautes \'Etudes Sci. Publ. Math. No. 80 (1994), 5-79 (1995).


%\bibitem[Si-1997]{si} C. Simpson, ``The Hodge filtration on nonabelian cohomology," Algebraic geometry?Santa Cruz 1995, 217?281, Proc. Sympos. Pure Math., 62, Part 2, Amer. Math. Soc., Providence, RI, 1997.



%\bibitem[Stacks]{stpr} The stacks project


%\bibitem[Su-1974]{su} H. Sumihiro, ``Equivariant completion,"  J. Math. Kyoto Univ. 14 (1974), 1-28.

\end{thebibliography}
\end{document}